\documentclass[12pt]{article}
\usepackage{amsmath,amsfonts,amssymb,amsthm,amscd}
\title{A note on  groups definable in the $p$-adic field}
\date{\today}
\author{Anand Pillay\thanks{Supported by NSF grants DMS-1360702,  DMS 1665035, and DMS 1760413}\\University of Notre Dame \and   Ningyuan Yao\thanks{Supported by NSFC grant 11601090 and Shangai Puijan Program 16PJC018}\\Fudan University }

\newtheorem{Theorem}{Theorem}[section]
\newtheorem{Proposition}[Theorem]{Proposition}

\newtheorem{Remark}[Theorem]{Remark}
\newtheorem{Lemma}[Theorem]{Lemma}

\newtheorem{Fact}[Theorem]{Fact}

\newtheorem{Question}[Theorem]{Question}

\newcommand{\Q}{\mathbb Q}

\begin{document}
\maketitle

\begin{abstract}   It is known \cite{HP} that a group $G$ definable in the field $\Q_{p}$ of $p$-adic numbers is definably locally isomorphic to the group $H(\Q_{p})$ of $p$-adic points of a (connected) algebraic group $H$ over $\Q_{p}$. We observe here that if $H$ is commutative then  $G$ is commutative-by-finite.  This shows in particular that any one-dimensional group $G$ definable in $\Q_{p}$ is commutative-by-finite.  This extends to groups definable in $p$-adically closed fields. We situate the results in a {\em geometric structures} environment. 

\end{abstract}

\section{Introduction and preliminaries}
We here consider analogous questions vis-a-vis the $p$-adic field (or more generally $p$-adically closed fields) to questions long studied for the real field (more generally real closed fields, $o$-minimal structures).  Namely the classification and description of definable groups.  We emphasize {\em definable}  rather than interpretable. These coincide in the real case because of elimination of imaginaries, but not in the $p$-adic case.  Common features  are that definable groups have naturally the structure of real and $p$-adic Lie groups  \cite{Pillay1}, \cite{Pillay2}, as well as being locally isomorphic to real/$p$-adic  algebraic groups \cite{HP}.  But from this point on, things diverge: the $p$-adic dimension on definable sets is less well-behaved, in particular definable groups are in general far from being definably connected, and arguments from the real case do not go through.

We prove here that when $G$ is locally commutative, equivalently the (connected)  algebraic group $H$ over $\Q_{p}$ such that $G$ is definably locally isomorphic to $H(\Q_{p})$, is commutative, then $G$ is commutative-by-finite.  When $G$ is one--dimensional (in the sense of $p$-adic dimension) then $H$ will be a connected algebraic group of algebraic-geometric dimension $1$, so commutative. Thus we deduce from our results that  one-dimensional groups definable in the $p$-adics are commutative-by-finite. 

$Th(\Q_{p})$ (i.e. the home sort) is $dp$-minimal (see \cite{Simon}), hence any one-dimensional group $G$ definable in $Q_{p}$ is also $dp$-minimal, so by a result of Simon \cite{Simon}, $G$ is commutative-by-bounded exponent. So we could deduce commutative-by-finiteness of one-dimensional $G$ definable in $\Q_{p}$  if we knew that there are no bounded exponent groups {\em interpretable}  in the $p$-adics other than finite groups.  This must be true but we do not know a proof at the moment.

The proof of our main result, which is in effect passing from being locally commutative to being commutative-by-finite,  is relatively soft, and we situate it in the context of geometric structures. 

A commutative-by-finite group $G$ is amenable (as a discrete group), hence  definably amenable with respect to any ambient structure in which $G$ happens to be definable. As also $Th(\Q_{p})$ is $NIP$ (see Section 4.2, \cite{Belair}), this puts us in a position to be able to apply some nice results from \cite{MOS} which refine the ``algebraic group configuration theorem" of \cite{HP}, and we (and others)  intend to carry this out in a future work. Even though the current paper is short we thought it makes sense as a ``stand-alone" paper as it is self-contained, the  methods are elementary,  and it may be useful for future work by the authors and others.

We use fairly basic model theory. We refer to the excellent survey \cite{Belair} as well as \cite{HP} and \cite{Onshuus-Pillay}  for the model theory of the $p$-adic field 
$(\Q_{p},+,\times, 0,1)$. In fact both \cite{HP} and \cite{Onshuus-Pillay} are also good references for the model theoretic background required for the current paper. A {\em geometric structure} (see Section 2 of \cite{HP})  is a one-sorted structure $M$ such that in any model $N$ of $Th(M)$, algebraic  closure satisfies exchange (so gives a so-called pregeometry on $N$) and there is a finite bound on the sizes of finite sets in uniformly definable families.  The structure $(\Q_{p},+,\times,0,1)$ is an example of a geometric structure (\cite{HP}, Proposition 2.11), as model-theoretic algebraic closure coincides with field-theoretic algebraic closure. 

In a geometric structure $M$, if $a$ is a finite tuple from $M$ and $B$ a subset of $M$ then $dim(a/B)$ denotes the size of a maximal algebraically independent over $B$ subtuple of $a$. If $M$ is saturated and $X$ is a $B$-definable subset of $M^{n}$ (where $B$ is finite) then $dim(X)= max\{dim(a/B):a\in X\}$.  It is important to know that when $M$ is $(\Q_{p},+,\times,0,1)$, and $X\subseteq M^{n}$ is definable, then its dimensioin in the above sense coincides with its ``topological dimension" , namely the greatest $k\leq n$ such that the image of $X$ under some projection from $M^{n}$ to $M^{k}$ contains an open set.

In one  of the general results below we make an assumption on the existence of $G^{0}$.  So we explain what this means, although it is discussed in the first section of \cite{Onshuus-Pillay}.  Work in a saturated structure ${\bar M}$ and suppose $G$ is definable in ${\bar M}$. Let $A$ be a small (of cardinality strictly less than the degree of saturation of ${\bar M}$) subset of ${\bar M}$ such that $G$ is defined over $A$. Then by $G^{0}_{A}$ we mean the intersection of all $A$-definable subgroups of $G$ of finite index. We say that {\em $G^{0}$ exists} if $G^{0}_{A}$ does not depend on $A$, namely cannot get smaller by increasing $A$.  Note that the existence of $G^{0}$ is equivalent to the nonexistence of an infinite uniformly definable family of subgroups of $G$ of some fixed finite index, which amounts to the $DCC$ on intersections of uniformly definable subgroups of (a given) finite index.  As mentioned in \cite{Onshuus-Pillay}, the existence of $G^{0}$ follows from the ambient theory having $NIP$. 

Finally let us state clearly the local isomorphism results alluded to earlier.

\begin{Fact} (\cite{Pillay2})   Let $G$ be a group definable in the field $\Q_{p}$. Then $G$ has definably the structure of a $p$-adic Lie group. Moreover if $G$ has dimension $k$  as a definable group then it has dimension $k$ as a $p$-adic Lie group.

\end{Fact}

\begin{Fact} (\cite{HP}) Suppose $G$ is a group definable in the field $\Q_{p}$. Consider $G$ with its topology given by Fact 1.1. Then there is a connected algebraic group $H$ over $\Q_{p}$, where the algebraic-geometric dimension of $H$ equals the dimension of $G$, and a definable homeomorphism $f$ between an open neigbourhood $U$ of the identity in $G$ and an open neighbourhood $V$  of the identity of (the $p$-adic Lie group) $H(\Q_{p})$ such that $f(ab) = f(a)f(b)$ whenever $a,b\in  U$ and $ab\in U$. 

\end{Fact}

\begin{Remark} (i) In Fact 1.2 we can actually choose $f$ to be a definable isomorphism (and homeomorphism) between definable open subgroups of $G$ and $H(\Q_{p})$, because any (definable) $p$-adic Lie group has a (definable) compact open subgroup, and a compact $p$-adic Lie group is profinite. 
\newline
(ii)  In fact both Fact 1.1 and Fact 1.2 hold (with appropriate definitions of definable Lie group over a $p$-adically closed field $K$) for groups $G$ definable in an elementary extension $K$ of $\Q_{p}$.  See Step 1 of the proof of Theorem 2.1 in \cite{HP2} in the real closed field situation which works word for word in the $p$-adically closed field case. 

\end{Remark}

\vspace{5mm}
\noindent
{\em Acknowledgements.}  Both authors would  like to thank the Institut Henri Poincar\'{e}, Paris, for its hospitality and support during the trimester on model theory in early 2018 when this work was done. The second author would like to thank the IHES, Orsay,  for its hospitality during the academic year 2017-18.  Both authors would like to thank Immi Halupczok for discussions, in particular for telling us some valued-field arguments around  finite index of centralizers  which we use in the paper  (although we situate it in a more general environment).

\section{Results}

We start with an easy lemma about geometric structures.

\begin{Lemma} Suppose $M$ is a geometric structure, $X\subseteq M^{n}$ a definable set of dimension $k$ say, and $f$ is a definable function from $X$ to $M^{m}$. Suppose that $f^{-1}(b)$ has dimension $k$ for all $b\in Im(f)$ then $Im(f)$ is finite.
\end{Lemma} 
\begin{proof} We may assume $M$ to be saturated and work over the parameters over which $X$ and $f$ are defined.  Suppose for a contradiction $Im(f)$ to be infinite. 
Then we can find $b\in Im(f)$ such that $dim(b)\geq 1$. As $dim(f^{-1}(b)) = k$ we can find $a\in f^{-1}(b)$ such that $dim(a/b) = k$. Hence by subadditivity, $dim(a,b) > k$. As $b\in dcl(a)$ it follows that $dim(a) >k$, contradicting that $a\in X$ and $dim(X) = k$. 
\end{proof} 

\begin{Remark} The conclusion of the lemma is weaker than stating that there is no definable equivalence relation on a definable $k$-dimensional set with infinitely many classes of dimension $k$. The latter is false in $\Q_{p}$

\end{Remark}

\begin{Proposition} Let $M$ be a geometric structure, and let  $G\subseteq M^{n}$ be a group definable in $M$ with $dim(G) = k$. Assume that (working in a saturated model) $G^{0}$ exists. Suppose that $G$ contains a definable subset $X$ of dimension $k$ such that $ab = ba$ for all $a,b\in X$. Then $G$ is commutative-by-finite, namely $G$ has a (definable) subgroup $H$ of finite index such that $H$ is commutative.
\end{Proposition}
\begin{proof} For $a\in X$, let $f_{a}$ be the (definable) function from $G$ to $G$ defined by $f_{a}(g) = gag^{-1}a^{-1}$. 
\newline
{\em Claim I.}  Fix $a\in X$.  For any $g_{1}, g_{2}\in G$, $f_{a}(g_{1}) = f_{a}(g_{2})$ iff $g_{1}C_{G}(a) = g_{2}C_{G}(a)$.
\newline
{\em Proof of Claim I.} 
\newline
This is obvious but we do it anyway. 
$g_{1}a g_{1}^{-1}a^{-1} = g_{2}ag_{2}^{-1}a^{-1}$ iff $g_{2}^{-1}g_{1}ag_{1}^{-1}g_{2} = a$, namely $g_{1}g_{2}^{-1}\in C_{G}(a)$.

\vspace{2mm}
\noindent
{\em Claim II.}  For $a\in X$, $Im(f_{a})$ is finite.
\newline
{\em Proof of Claim II.}  
\newline
Note that $C_{G}(a)$, the centralizer of $a$ in $G$ contains $X$ (by assumption) so has dimension $k$. Hence for any $g\in G$, $dim(gC_{G}(a)) = k$.
By the right implies left implication in Claim I, $f_{a}$ is constant on $gC_{G}(a)$ for all $g\in G$. So we conclude by Lemma 2.1.

\vspace{2mm}
\noindent
{\em Claim III.} For any $a\in X$, $C_{G}(a)$ has finite index in $G$.
\newline
{\em Proof of Claim III.}  
\newline
By Claim I, $Im(f)$ is in bijection with $G/C_{G}(a)$, so by Claim II, $C_{G}(a)$ has finite index in $G$. 

\vspace{5mm}
\noindent
We may assume that we have been working in a saturated model $M$. So by compactness there is a finite bound on the index of $C_{G}(a)$ in $G$ for $a\in X$. Our assumption that $G^{0}$ exists (as discussed in the previous section)  implies that $C_{G}(X) = \cap_{a\in X}C_{G}(a)$ is a finite subintersection, so is a definable subgroup $H$ say of $G$, of finite index.

Now for any $a\in H$, $C_{G}(a)$ contains $X$, so has dimension $k$. So repeating Claims I, II, and III, for $a$ in $H$ rather than $X$ we conclude that $C_{G}(a)$ has finite index in $G$ for all $a\in H$, and so again that $C_{G}(H)$ has finite index in $G$.  So $H\cap C_{G}(H)$ has finite index in $G$ and is  commutative.
\end{proof}

\begin{Remark} One really  needs only a finite-valued subadditive dimension on types of real tuples, for Lemma 2.1 and Proposition 2.3. And for Proposition 2.3, one only needs a definable set $X\subseteq G$ of dimension $k$ such that for all $a\in X$, $C_{G}(a)$ has dimension $k$. 
\end{Remark}

In any case, from  Proposition 2.3 we conclude:
\begin{Theorem}  (i) Let $G$ be a group definable in $\Q_{p}$. Let $H$ be a connected algebraic group over $\Q_{p}$ as in Fact 1.1. Suppose that $H$ is commutative, then $G$ is commutative-by-finite.
\newline
(ii)  Let $G$ be a group of dimension $1$ definable in $\Q_{p}$. Then $G$ is commutative-by-finite. 
\end{Theorem}
\begin{proof} (i)  Let $U$, $V$ be definable open neighbourhoods of the identity of $G$, $H(\Q_{p})$ respectively and $f:U\to V$ as given by Fact 1.2. By choosing a smaller definable open neigbourhood $U_{1}$ of the identity contained in $G$ such that $ab\in U$ for $a,b\in U_{1}$, we see from the assumptions that $ab = ba$ for $a,b\in U_{1}$. Now the $p$-adic topological) dimension of $U_{1}$ coincides with that of $G$, but both coincide with the dimension with respect to $\Q_{p}$ as a geometric structure. So, bearing in mind that $G^{0}$ exists (working in a saturated model), as remarked in the introduction. We can apply Proposition 2.3 to conclude that $G$ is commutative-by-profinite.
\newline
(ii)  If $G$ has dimension $1$ then by Fact 1.2 the connected algebraic group $H$ has dimension $1$ as an algebraic group, so is commutative. So part (i) implies.

\end{proof} 

\begin{Remark}  By Remark 1.3 (ii),  Theorem 2.4 goes through for groups definable in arbitrary $p$-adically closed fields (i.e. models of $Th(\Q_{p})$). 
\end{Remark}

\begin{Question} (i)  Do we need the assumption that $G^{0}$ exists in Proposition 2.3?  We presume yes, namely there is a counterexample without it.
\newline
(ii)  Let $F$ be a geometric field in the sense of  Definition 2.9 of \cite{HP}. Let $G$ be a group definable in $F$ and let $H$ be a connected algebraic group over $F$ given by Proposition 3.1' of \cite{HP}. Suppose $H$ is commutative. Can one find a definable subset $X$ of $G$ of dimension equal to $dim(G)$ such that for all $a\in X$, $C_{G}(a)$ has dimension equal to $dim(G)$?
\end{Question}

\end{document}